\documentclass[a4paper, 11pt, reqno]{amsart}
\usepackage[usenames,dvipsnames]{color}
\usepackage{amssymb, amsmath, latexsym, graphics, graphicx}

\definecolor{lightgrey}{gray}{0.75}
\newtheorem{theorem}{Theorem}[section]
\newtheorem{lemma}[theorem]{Lemma}
\newtheorem{corollary}[theorem]{Corollary}

\newtheorem{proposition}[theorem]{Proposition}

\begin{document}
\title{Convexity of tropical polytopes}
\maketitle
\begin{center}
MARIANNE JOHNSON\footnote{Email \texttt{Marianne.Johnson@maths.manchester.ac.uk}.}  and MARK KAMBITES\footnote{
Email \texttt{Mark.Kambites@manchester.ac.uk}.}

    \medskip

    School of Mathematics, \ University of Manchester, \\
    Manchester M13 9PL, \ England.

 \date{\today}

\thanks{}
\end{center}
\numberwithin{equation}{section}

\begin{abstract}
We study the relationship between min-plus, max-plus and Euclidean
convexity for subsets of $\mathbb{R}^n$. We introduce a
construction which associates to any max-plus convex
set with compact projectivisation a canonical
matrix called its \textit{dominator}. The dominator is a Kleene star
whose max-plus column space is the min-plus convex hull of the original set. We
apply this to show that a set which is any two of (i) a max-plus polytope,
(ii) a min-plus polytope and (iii) a Euclidean polytope must also be the
third.
In particular, these results answer a question of Sergeev, Schneider and
Butkovi\v{c} \cite{Sergeev09b} and show that row spaces of tropical Kleene
star matrices are exactly the ``polytropes'' studied by Joswig and Kulas \cite{Joswig10}.
\end{abstract}

The notion of \textit{tropical convexity} (also known as \textit{max-plus}
or \textit{min-plus} convexity) has long played an important role in max-plus
algebra and its numerous application areas (see for example
\cite{Butkovic10,Butkovic07}). More recently, applications have emerged in
areas of pure mathematics as diverse as algebraic geometry
\cite{Develin04} and semigroup theory \cite{K_tropd,K_tropj}.

Recall that a subset of $\mathbb{R}^n$ is called \textit{max-plus convex} if it is
closed under the operations of componentwise maximum (``max-plus sum'')
and of adding a fixed value to each component (``tropical scaling'').
A \textit{max-plus polytope} is a non-empty max-plus convex set which is generated
under these operations by finitely many of its elements;  max-plus polytopes
are exactly the row spaces (or column spaces) of matrices over the max-plus
semiring. There are obvious dual notions of min-plus convexity and min-plus
polytopes. Min-plus and max-plus polytopes are sometimes called
\textit{tropical polytopes}\footnote{Typically one fixes upon either the
``min convention'' or the ``max convention'' and uses terms such as
``tropically convex'' and ``tropical polytope'' to refer to min-plus or
max-plus as appropriate. A key feature of this paper is that we study
the relationship \textit{between} min-plus and max-plus convexity, so
for the avoidance of ambiguity we will tend not to use the word
``tropical''.}.

The \textit{projectivisation} of a max-plus polytope
is the set of orbits of its points
under the action of tropical scaling.
Since any point can be scaled to put $0$ in the first coordinate (say),
restricting the polytope to points with first coordinate $0$ gives a
cross-section of the scaling orbits, and hence a natural
identification of the projectivisation with a subset of
$\mathbb{R}^{n-1}$. A subset which arises from a max-plus polytope
in this way we term a \textit{projective max-plus polytope}.
In general, a projective max-plus polytope is a compact Euclidean polyhedral
complex in $\mathbb{R}^{n-1}$; it may or may not be a convex set in the ordinary
Euclidean metric. Joswig and Kulas \cite{Joswig10} studied the
class of projective max-plus polytopes\footnote{Formally speaking they studied
projective \textit{min-plus} polytopes, but the difference is immaterial because
\textit{negation} of vectors
forms a trivial duality between min-plus and max-plus convex sets while preserving
Euclidean convexity.} which \textit{are} Euclidean convex, and hence also
Euclidean polytopes. These sets, which they term \textit{polytropes}, turn out to have numerous
interesting properties and applications (see for example \cite{delaPuente13, Tran14, Werner14}).

At around the same time, Sergeev, Schneider and Butkovi\v{c} \cite{Sergeev09b}
studied the tropical polytopes arising as row spaces of \textit{Kleene stars}. Kleene stars are a class of particularly well-behaved idempotent max-plus
matrices which play a vital role in almost all aspects of max-plus
algebra (see Section~\ref{sect:resid} below for a definition and \cite{Butkovic10}
for a comprehensive
introduction). They prove, among many other interesting things, that the column
space of a max-plus\footnote{In fact \cite{Sergeev09b} is written in the language
of max-\textit{times} algebra and geometry; this is exactly equivalent
to max-plus via the logarithm map, save for the fact that they work with a
$0$ element, which corresponds to a $-\infty$ element in the max-plus case.
For simplicity we work here without $-\infty$, but the results of \cite{Sergeev09b}
all specialise to apply in this case. The concept corresponding to Euclidean convexity
in the max-times setting is \textit{log-convexity}.}
Kleene star is always
Euclidean convex \cite[Propositions~2.6 and 3.1]{Sergeev09b}; since such column spaces are max-plus polytopes by
definition, this means they are all examples of
 Joswig-Kulas ``polytropes''. The question of whether the converse holds
was posed (with a slightly differently phrasing) in \cite[p.~2400]{Sergeev09b}:
is a Euclidean-convex
max-plus polytope
necessarily the column space of a Kleene star? One aim of the present
paper is to answer this question in the affirmative, establishing
that ``polytropes'' are exactly column (and row) spaces of Kleene stars.
This unifies
the research in \cite{Joswig10} and \cite{Sergeev09b}, and opens up the possibility
of using methods from each to address the kind of questions considered
in the other.

(Note that in the abstract and introduction of \cite{delaPuente13} it is stated that a matrix with $0$s on the diagonal has
Euclidean-convex max-plus column space if and only if it is a max-plus Kleene star. Some of our results would follow from
this claim, but in fact the claim is false: for example the matrix
{\small $\left( \begin{matrix} 0 & 1 \\ 1 & 0 \end{matrix} \right)$} has $0$s on the diagonal and
Euclidean-convex max-plus column space but is not a max-plus Kleene star.
The main result \cite[Theorem~2.1]{delaPuente13} is correct, but does not
suffice to establish either the more general claim in the abstract, or our own
results.)

The key to our approach is to study min-plus closures of max-plus
polytopes (or more generally, max-plus convex sets with compact
projectivisation).
We
introduce (in Section~\ref{sect:Kleene}) an elementary, but
surprisingly powerful, construction which
canonically associates a matrix to each such set.
This
matrix, which we term the \textit{dominator} of the set, is always a
max-plus Kleene star, and we show that its max-plus column space is exactly the
min-plus convex hull of the original set. In particular, if the chosen set
was already min-plus convex then it is exactly the column space of the dominator.
Combining this result with its min/max dual and with results of \cite{Sergeev09b}
yields two elegant and useful characterisations of
min-plus convexity for max-plus convex sets with compact projectivisation:

\newpage

\begingroup
\setcounter{theorem}{0} 
\renewcommand\thetheorem{\Alph{theorem}}
\begin{theorem}\label{thm:mainresult}
Let $P\subseteq \mathbb{R}^n$ be a max-plus convex set with compact
projectivisation.
Then the following are equivalent:
\begin{itemize}
\item[(i)] $P$ is min-plus convex;
\item[(ii)] $P$ is the max-plus column space of a max-plus Kleene star;
\item[(iii)] $P$ is the min-plus column space of a min-plus Kleene star.
\end{itemize}
\end{theorem}
Of course Theorem~\ref{thm:mainresult} also has a natural dual statement,
obtained by exchanging ``min'' and ``max'' throughout.

A non-trivial corollary of Theorem~\ref{thm:mainresult} is that the
min-plus convex hull of a max-plus polytope is always a max-plus polytope.
(The fact that such a set is max-plus convex is an easy consequence of
the distributivity of maximum over minimum, but the fact that it must
be max-plus finitely generated is less obvious.)	

We then turn our attention to Euclidean convexity,
showing that for a max-plus polytope, being min-plus convex and
being Euclidean convex are in fact equivalent.
 Indeed, combining this
statement with its max-min dual we see that (modulo issues of finite
generation), any two of the three notions of convexity under
consideration together imply the third:
\begin{theorem}\label{thm:mainresultfg}
Let $P \subseteq \mathbb{R}^n$. If any two of the following statements
hold, then so does the remaining one:
\begin{itemize}
\item[(i)] $P$ is a max-plus polytope;
\item[(ii)] $P$ is a min-plus polytope;
\item[(iii)] $P$ is Euclidean convex.
\end{itemize}
\end{theorem}
Notice that the hypothesis and condition (iii) are invariant under exchanging
min and max, while conditions (i) and (ii) are dual to each other. As a
consequence of this, only two implications need to be
proved to establish the theorem. One of these implications is proved via
Kleene stars, by combining Theorem~\ref{thm:mainresult} with a result of
Sergeev, Schneider and Butkovi\v{c} \cite{Sergeev09b}; the other is established by an elementary, but quite
technical, direct argument.

In contrast to Theorem~\ref{thm:mainresult}, the finite generation hypotheses
implied by the word ``polytope'' in conditions (i) and (ii) of
Theorem~\ref{thm:mainresultfg} are essential, and cannot be relaxed to
assume only compactness of the projectivisation.
Indeed, we shall see below (Section~\ref{sect:maintheorems}) an example of a
Euclidean convex, max-plus convex set with compact projectivisation which is
nevertheless not min-plus convex.

Theorem~\ref{thm:mainresult}, its dual and Theorem~\ref{thm:mainresultfg}
can be combined in various ways to show the equivalence of various other
combinations of conditions. Probably
the most interesting example is the promised proof that the ``polytropes'' of Joswig
and Kulas \cite{Joswig10} are exactly the column spaces of Kleene stars as studied by Sergeev,
Schneider and Butkovi\v{c} \cite{Sergeev09b}:
\begin{corollary}
A max-plus [min-plus] polytope is a Euclidean polytope if and only if it is
the column space of a max-plus [respectively, min-plus] Kleene star.
\end{corollary}
\endgroup
\setcounter{theorem}{0}

\section{Residuation and domination}
\label{sect:resid}

We write $a \oplus b := \max(a,b)$, $a \boxplus b: = \min(a,b)$ and $a
\otimes b:=a + b$ for all $a,b \in \mathbb{R}$. The operations $\oplus$ and
$\otimes$ (respectively $\boxplus$ and $\otimes$) give $\mathbb{R}$ the structure of a
\textit{semiring}, called the \textit{max-plus} (\textit{min-plus}) \textit{semiring}.
The two semiring structures each induce natural multiplication operations for
suitably-sized matrices over $\mathbb{R}$, which we denote by $\otimes$ and $\boxtimes$
respectively.

For $x \in \mathbb{R}^n$ we write $x_i$ to denote the $i$th entry of $x$.
There is a natural partial order on $\mathbb{R}^n$ defined by $x \leq y$ if and only if $x_i \leq y_i$ for all $i$.
The operations $\oplus$ and $\boxplus$ extend entrywise to elements of $\mathbb{R}^n$
in the obvious way, and indeed are the least upper bound and greatest lower
bound operations respectively, with regard to the order.
There is also a
\textit{tropical scaling} action of $\mathbb{R}$ on $\mathbb{R}^n$ given by
$$\lambda \otimes (x_1, \dots, x_n) \ = \ (\lambda \otimes x_1, \dots, \lambda \otimes x_n) \ = \ (x_1 + \lambda, \dots, x_n + \lambda).$$
These operations make $\mathbb{R}^n$ into a module over both the max-plus
and the min-plus semiring.

A subset of $\mathbb{R}^n$ is called \textit{max-plus
convex}\footnote{The term \textit{convex} comes from an
alternative characterisation in terms of \textit{tropical
line segments}; see \cite{Develin04}.} (respectively,
min-plus convex) if it is closed under $\oplus$ (respectively $\boxplus$)
and tropical scaling. The \textit{max-plus convex hull} of any subset
$P \subseteq \mathbb{R}^n$ is the smallest max-plus convex set containing $P$, or
equivalently the set of all finite max-plus linear combinations of elements
in $P$, or the submodule of $\mathbb{R}^n$ (viewed as a module over the
max-plus semiring) generated by $P$; we denote it $\textrm{Span}_\oplus(P)$.
The  \textit{min-plus convex hull} $\textrm{Span}_\boxplus(P)$ is defined
dually. If $M$ is a matrix over $\mathbb{R}$ then its \textit{max-plus
column space} $\textrm{Col}_\oplus(M)$ is the max-plus convex hull of its
columns; the \textit{min-plus column space} $\textrm{Col}_\boxplus(M)$,
\textit{max-plus row space} $\textrm{Row}_\oplus(M)$ and
\textit{min-plus row space} $\textrm{Row}_\boxplus(M)$ are defined
similarly in the obvious way.

For all $x, y \in \mathbb{R}^n$ we define
$$\langle x \mid y \rangle = {\rm max} \{\lambda \in \mathbb{R}: \lambda \otimes x \leq y\} = {\rm min}_i\{y_i-x_i\}.$$
This operation is a \textit{residuation} operator in the sense of residuation theory
\cite{Blyth72}, and has been extensively applied in max-plus algebra and geometry (see
for example \cite{Baccelli92,Cohen04,Gaubert07b,K_tropd}).
We record a number of useful properties, which the reader can easily verify using the definition above:
\begin{equation}
\langle \lambda \otimes x\mid y \rangle \ = \ -\lambda \otimes \langle x\mid y \rangle \ = \ \langle x\mid -\lambda \otimes y \rangle
\label{eqn:scale}
\end{equation}
\begin{align}
\langle x \mid y \rangle \boxplus \langle x' \mid y \rangle \ &= \ \langle x \oplus x'\mid y \rangle \ &\leq \ \langle x\mid y \rangle \ &\leq \ \langle x\mid y \oplus y'\rangle
\label{eqn:oplus} \\
\langle x \mid y \rangle \boxplus \langle x \mid y' \rangle \ &= \langle x \mid y \boxplus y' \rangle
\ &\leq \ \langle x\mid y\rangle \ &\leq \ \langle x \boxplus x'\mid y \rangle
\label{eqn:boxplus}
\end{align}

We say that $x$ \emph{dominates} $y$ in position $i$ if $\langle x \mid y \rangle = y_i-x_i$. It follows from \eqref{eqn:scale} that domination is scale invariant, that is, if $x$ dominates $y$ in position $i$, then $\lambda \otimes x$ dominates $\mu \otimes y$ in position $i$ for all $\lambda, \mu \in \mathbb{R}$. We denote the set of all elements of $\mathbb{R}^n$ dominated by $x$ in position $i$ by ${\rm Dom}_i(x)$.  If $P$ is a subset of $\mathbb{R}^n$, we say that $x$ \emph{dominates} $P$ in position $i$ if $P \subseteq{\rm Dom}_i(x)$.

\begin{lemma}
\label{lem:domi}
For each $x \in \mathbb{R}^n$ and each coordinate $i$, the set
${\rm Dom}_i(x)$ is max-plus convex, min-plus
convex and Euclidean convex.
\end{lemma}

\begin{proof}
Since domination is scale invariant we see that ${\rm Dom}_i(x)$ is closed under tropical scaling. Thus, in order to show that ${\rm Dom}_i(x)$ is max-plus and min-plus convex, it suffices to show that $u \oplus v, u \boxplus v \in {\rm Dom}_i(x)$ for all $u,v \in {\rm Dom}_i(x)$.

It follows from (1.2) that
$$\langle x \mid u\oplus v \rangle \geq \langle x \mid
u \rangle \oplus  \langle x \mid v \rangle.$$ By assumption the
latter is equal to $(u_i-x_i) \oplus
(v_i-x_i)$, or in other words, $(u \oplus v)_i - x_i$.
Conversely, it is immediate from the definition of the bracket
that
$$(u\oplus v)_i - x_i \geq {\rm min}_k((u\oplus v)_k - x_k) = \langle x \mid u \oplus v \rangle$$
so we have
$$(u\oplus v)_i - x_i = \langle x \mid u \oplus v \rangle$$
as required to show that $u \oplus v$ is dominated by $x$ in position $i$.

On the other hand by \eqref{eqn:boxplus} we have
$$\langle x \mid u \boxplus v \rangle = \langle x \mid u \rangle \boxplus \langle x \mid v \rangle,$$
giving
$$\langle x \mid u \boxplus v \rangle = (u_i-x_i) \boxplus (v_i - x_i) = (u \boxplus v)_i - x_i,$$
showing that $u \boxplus v$ is dominated by $x$ in position $i$.

Finally, let $u, v \in {\rm Dom}_i(x)$, $t \in [0,1]$ and let $z=tu + (1-t)v$ be a point on the line segment between $u$ and $v$. It follows from the fact that $u, v \in {\rm Dom}_i(x)$ that
$$u_i \leq u_j + (x_i-x_j),\;\;\; v_i \leq v_j + (x_i-x_j), \;\;\; \mbox{ for all } j.$$
Since $0 \leq t, (1-t)$ for all $j$ we have
\begin{eqnarray*}
z_i = tu_i + (1-t)v_i &\leq& t(u_j + (x_i-x_j)) + (1-t)(v_j + (x_i-x_j))\\
                      &\leq& tu_j + (1-t)v_j + (x_i-x_j)\\
                      &\leq& z_j + (x_i-x_j).
\end{eqnarray*}
In other words, $x$ dominates $z$ in position $i$.
\end{proof}
The fact that ${\rm Dom}_i(x)$ is Euclidean convex can also be deduced from the
work of Develin and Sturmfels \cite[Lemma 10]{Develin04}.

\begin{lemma}
\label{lem:domii}
Let $S$ be a subset of $\mathbb{R}^n$ and $P$ its max-plus (or min-plus
or Euclidean) convex hull.
Then $x$ dominates
$P$ in position $i$ if and only if $x$ dominates $S$ in position $i$.
\end{lemma}
\begin{proof}
One implication is immediate from the definition of domination, while
the other follows straight from Lemma~\ref{lem:domi}.
\end{proof}

\begin{lemma}
\label{lem:int}
Suppose that $c_1, \ldots, c_n \in \mathbb{R}^n$ are such that for all $i$ and $j$, $c_i$ dominates $c_j$ in position $i$. Then the intersection $\bigcap_i {\rm Dom}_i(c_i)$ is a max-plus \emph{polytope} (with generating set $c_1, \ldots, c_n$) that is min-plus and Euclidean convex.
\end{lemma}

\begin{proof}
Let $D =\bigcap_i {\rm Dom}_i(c_i)$. By
Lemma \ref{lem:domi}, $D$ is an intersection of sets which are max-plus,
min-plus and Euclidean convex, and so is itself max-plus, min-plus and
Euclidean convex. It follows straight from the hypothesis that
 $c_1, \ldots, c_n \in D$ and, since $D$ is max-plus convex, also
that
$${\rm Span}_{\oplus}(c_1, \ldots, c_n) \subseteq D.$$
Moreover,
if $y \in D$, that is, $y \in {\rm Dom}_i(c_i)$ for all $i$, then it
is easy to see that for all $i$ the inequality $\langle c_i \mid y \rangle \otimes c_i \leq y$ holds, with equality in position $i$. Thus
we have
$$y = \bigoplus_{i=1}^n  \langle c_i \mid y \rangle \otimes c_i \in {\rm Span}_{\oplus}(c_1, \ldots, c_n)$$
so that
 $D \subseteq {\rm Span}_{\oplus}(c_1, \ldots, c_n) \subseteq D$.
\end{proof}

\section{Kleene stars and dominator matrices}
\label{sect:Kleene}

Recall that a matrix $A \in M_n(\mathbb{R})$ is called a \textit{(max-plus) Kleene
star}\footnote{This is a slightly non-standard definition, which we prefer
here since it avoids introducing additional terminology for which we have
no further use. For the standard definition see for example \cite{Butkovic10};
for equivalence of the two definitions see for example
\cite[Proposition~2.1]{Sergeev09b}.} if $A \otimes A = A$ (that is, $A$ is
idempotent) and all diagonal entries of $A$ are $0$.
 Kleene stars have numerous fascinating properties
and important applications, and have long played a central role in max-plus
algebra; see for example \cite{Butkovic10} for a full introduction.

The following lemma provides a connection between Kleene stars and the
concept of domination introduced in Section~\ref{sect:resid}.

\begin{lemma}
\label{lem:Kleene}
Let $K \in M_n(\mathbb{R})$ be a max-plus Kleene star, with columns $c_1, \ldots, c_n$. Then
$${\rm Col}_\oplus(K)= \bigcap_i {\rm Dom}_i(c_i).$$
In particular, ${\rm Col}_\oplus(K)$ is min-plus and Euclidean convex.
\end{lemma}

\begin{proof}
Since $K$ is a Kleene star, it follows from \cite[Lemma 5.3(ii)]{K_finitemetric} that $(c_j)_i = \langle c_i \mid c_j \rangle $ for all $i$ and $j$.  Thus we have,
$$\langle c_i \mid c_j \rangle = (c_j)_i - (c_i)_i,$$
hence showing that each $c_i$ dominates each $c_j$ in position $i$.
It then follows from Lemma~\ref{lem:int} that
$\bigcap_i {\rm Dom}_i(c_i)$ is the max-plus polytope generated by
the columns $c_i$, (that is, the column space of $K$) and that this polytope
is min-plus and Euclidean convex.
\end{proof}

The fact that the column space of a max-plus Kleene star is Euclidean
convex was first proved (in the language of max-times algebra, and working
with a zero element) by Sergeev, Schneider and Butkovi\v{c}
\cite[Propositions~2.6 and 3.1]{Sergeev09b}.

For our next lemma, we shall need some facts about the duality between
the row and column spaces of a tropical matrix.
Given a matrix $A \in M_n(\mathbb{R})$, the max-plus duality maps of $A$ are defined as follows
\begin{eqnarray*}
\rho_A : {\rm Row}_\oplus(A) \rightarrow {\rm Col}_\oplus(A),&& \;\;\; r \mapsto A \otimes (-r)^T \mbox{ for all } r \in {\rm Row}_\oplus(A),\\
\chi_A : {\rm Col}_\oplus(A) \rightarrow {\rm Row}_\oplus(A),&& \;\;\; c \mapsto (-c)^T\otimes A \mbox{ for all } c \in {\rm Col}_\oplus(A).
\end{eqnarray*}
The maps $\rho_A$ and $\chi_A$ are mutually inverse bijections
between the max-plus row space and the max-plus column space of
the matrix $A$; they have many interesting properties --- see for example
\cite{Cohen04,Develin04,K_tropd}. When the matrix $A$ is a
Kleene star, it turns out that the duality maps have a particularly nice form:
\begin{lemma}
\label{lem:dual}
Let $K$ be a max-plus Kleene star. Then
\begin{itemize}
\item[(i)] The negated matrix $-K$ is a min-plus Kleene star.
\item[(ii)] $\chi_K(x)=-x^T$ for all $x \in {\rm Col}_\oplus(K)$.
\item[(iii)] $\rho_K(y)=-y^T$ for all $y \in {\rm Row}_\oplus(K)$.
\item[(iv)] ${\rm Col}_\oplus(K) = {\rm Col}_\boxplus(-K^T).$
\item[(v)] $ {\rm Row}_\oplus(K) = {\rm Row}_\boxplus(-K^T).$
\end{itemize}
\end{lemma}
\begin{proof}
(i) This is immediate from the fact that negation is an isomorphism
between the max-plus and min-plus semirings.

(ii) Let $x \in {\rm Col}_\oplus(K)$.
By Lemma~\ref{lem:Kleene}, for every $i$ we have that
the $i$th column of $K$ dominates $x$ in position $i$. By definition this means
$x_i - K_{i,i} \leq x_j - K_{j,i}$ for all $i$ and $j$, or by negating and rewriting
in tropical notation,
$$(-x)_i \otimes K_{i,i} \ \geq \ (-x)_j \otimes K_{j,i}$$
for all $i$ and $j$
But $K$ is a Kleene star so $K_{i,i} = 0$ and hence
$$(\chi_K(x))_i = \left( (-x)^T \otimes K \right)_i = \bigoplus_{j=1}^{n} (-x)_j \otimes K_{j,i} = (-x)_i \otimes K_{i,i} = (-x)_i$$
for all $i$, that is, $\chi_K(x)=-x^T$.

(iii) The proof is dual to (ii).

(iv) It follows from (ii) and (iii) that the max-plus column space
$\textrm{Col}_\oplus(K)$ is exactly the negation of the max-plus row space.
Because negation is an isomorphism between the max-plus and min-plus
semirings, this means the column space is the min-plus convex hull of
the negated rows of $K$, in other words, of the columns of $-K^T$.
But this is by definition $\textrm{Col}_\boxplus(-K^T)$.

(v) The proof is dual to (iv).
\end{proof}

\begin{lemma}
\label{lem:minmax}
The min-plus convex hull of a max-plus convex set is max-plus convex.
\end{lemma}

\begin{proof}
Let $P \subseteq \mathbb{R}^n $ be a max-plus convex set and let $u,v$ be two
elements of the min-plus closure $\textrm{Span}_{\boxplus}(P)$. It suffices
to show
that the max-plus sum $u \oplus v$ is contained in $\textrm{Span}_{\boxplus}(P)$.
Since $u, v \in \textrm{Span}_{\boxplus}(P)$ and $P$ is already closed under tropical
scaling, we can write $u = u_1 \boxplus \cdots \boxplus u_k$ and
$v = v_1 \boxplus \cdots \boxplus v_m$ for some $u_1, \ldots, u_k, v_1, \ldots, v_m \in P$. Then it is easy to check (using distributivity of max over min) that
\begin{eqnarray*}
u \oplus v &=& ((u_1 \oplus v_1) \boxplus \cdots \boxplus (u_k \oplus v_1)) \boxplus \cdots \boxplus ((u_1 \oplus v_m) \boxplus \cdots \boxplus (u_k \oplus v_m)).
\end{eqnarray*}
Since each max-plus sum $u_i \oplus v_j$ is an element of $P$, it follows that $u \oplus v$ is in $\textrm{Span}_{\boxplus}(P)$.
\end{proof}

\begin{lemma}
\label{lem:glb}
Let $P \subseteq \mathbb{R}^n$ be a non-empty max-plus convex set with compact
projectivisation.
\begin{itemize}
\item[(i)] For each coordinate $i$, the set $P_i = \{u \in P \mid u_i \geq 0\}$ has a unique greatest lower bound $d_i$. The element $d_i$ lies in the min-plus convex hull of $P$ and the $i$th component of $d_i$ is 0.
\item[(ii)]  If $x$ is in the min-plus convex hull of $P$, then $d_i$ dominates $x$ in position $i$, that is, $\langle d_i \mid x \rangle = x_i$ and hence
$$x = \bigoplus_i \langle d_i \mid x \rangle \otimes d_i.$$
\end{itemize}
\end{lemma}

\begin{proof}
(i) Fix $i$. For each $j \neq i$, it follows from compactness of the
projectivisation that the set $\{u_j \mid u \in P_i\}$ is
topologically closed and bounded below,
 and hence contains a minimum element. Thus, we may choose
$w_j \in P_i$ satisfying
$$(w_j)_i \geq 0 \mbox{ and } (w_j)_j = {\rm min}\{u_j \mid u \in P_i\}.$$
Note that $(w_j)_i = 0$; indeed, if not then
$(-(w_j)_i) \otimes w_j$ would be in $P_i$ with $j$th coordinate strictly
smaller than that of $w_j$, giving a contradiction.

Now let $d_i$ be the min-plus sum of the elements
$w_j$. Then the $i$th component of $d_i$ is 0. We claim that $d_i$ is the
greatest lower bound of $P_i$. By construction, $d_i \leq w_j$ for each
$j \neq i$, so by the definition of $w_j$ we have
$$(d_i)_j \leq (w_j)_j \leq u_j \textrm{ for all } u \in P_i \textrm{ and } j \neq i.$$
Moreover, using the definition of $P_i$ we have
$$(d_i)_i = 0 \leq u_i \textrm{ for all } u \in P_i.$$
Thus, $d_i$ is a lower bound for $P_i$.
If $y$ is any other lower bound for $P_i$, then in particular
$y \leq w_j$ for all $j \neq i$ and hence $y \leq d_i$.

(ii) Suppose first that $x \in P$. Consider the element $y = (-x_i) \otimes x$. Then $y_i = 0$
and $y \in P$ so $y \in P_i$ and hence $d_i \leq y$. Moreover,
$(d_i)_i = 0 = y_i$, so $\langle d_i \mid y \rangle = 0$ and using (1.1) we have
$$\langle d_i \mid x \rangle = \langle d_i \mid x_i \otimes y \rangle = x_i \otimes 0 = x_i$$
as required to show that $d_i$ dominates $x$ in position $i$. Thus, $d_i$
dominates $P$ in position $i$, and hence by Lemma~\ref{lem:domii}
$d_i$ dominates the min-plus convex hull of $P$.

Now if $x$ is in the min-plus convex hull of $P$ then
since $\langle d_i \mid x \rangle \otimes d_i \leq x$ and $\langle d_i \mid x \rangle \otimes (d_i)_i = \langle d_i \mid x \rangle =  x_i$, it follows that
$$x = \bigoplus_i \langle d_i \mid x \rangle \otimes d_i.$$
\end{proof}

The previous lemma motivates a key definition.
Let $P$ be a  max-plus convex subset of $\mathbb{R}^n$ with compact projectivisation. We define the
\emph{(min-plus) dominator} matrix of $P$, denoted $D_P$, to be the matrix whose
$i$th column is the greatest lower bound of the set $\{u \in P: u_i \geq 0\}$.
There is of course a natural dual concept of the \textit{max-plus} dominator of a
closed \textit{min-plus} convex set in $\mathbb{R}^n$.

\begin{proposition}
\label{prop:dommat}
Let $P \subseteq \mathbb{R}^n$ be max-plus convex with compact
projectivisation,
and let $D_P$ be the min-plus dominator matrix of $P$.
Then $D_P$ is a max-plus Kleene star, with max-plus column space exactly the 
min-plus convex hull of $P$. (In particular, the min-plus convex hull of
$P$ is a max-plus polytope.)
\end{proposition}

\begin{proof}
Let $P'$
denote the min-plus convex hull $\textrm{Span}_{\boxplus}(P)$.
By Lemma \ref{lem:glb}(i), each column $d_i$ is contained in $P'$.
Thus by Lemma \ref{lem:minmax} we see
that ${\rm Col}_\oplus(D_P) \subseteq P'$. On the other hand,
it follows from Lemma \ref{lem:glb}(ii) that
$P' \subseteq {\rm Col}_\oplus(D_P)$. Thus $P' = {\rm Col}_\oplus(D_P)$.

Now since each column $d_i$ lies in $P'$, applying
Lemma \ref{lem:glb} to these elements shows that $d_i$ dominates
$d_j$ in position $i$ for all $i$ and $j$. Thus by
Lemma \ref{lem:int} we see that ${\rm Col}_\oplus(D_P) = \bigcap_i {\rm Dom}_i(d_i)$.

In particular $P'$, being the max-plus column space of a matrix,
is a max-plus polytope, and hence has compact projectivisation, so
Lemma \ref{lem:glb}
also applies with $P$ replaced by $P'$. By Lemma \ref{lem:glb}(i) the set
$$P'_i = \{u \in P' \mid u_i \geq 0\}$$
has a greatest lower bound $f_i$, where $(f_i)_i=0$ and $f_i$ is
contained in the min-plus convex hull of $P'$, that is,
in $P'$ itself.
Hence by Lemma \ref{lem:glb}(ii), $\langle d_i \mid f_i \rangle = (f_i)_i=0$ and $\langle f_i \mid d_i \rangle = (d_i)_i=0$, giving $f_i = d_i$. In other words, $d_i$ is the greatest lower bound
of the set $P'_i$.

Now fix $j$ and for each $k$ let $y_k = -(d_j)_k \otimes d_j$. Then $y_k \in {\rm Col}_\oplus(D_P) = P'$ and $(y_k)_k = 0$, so $y_k \in P'_i$ and hence $d_k \leq y_k$. In particular this gives
$$(d_k)_i \leq -(d_j)_k \otimes (d_j)_i =(y_k)_i \mbox{ for all } i,k,$$
and hence
$$(d_k)_i +(d_j)_k \leq (d_j)_i \mbox{ for all } i,k.$$
In other words,
$$(D_P \otimes D_P)_{i,j} = \bigoplus_k (D_P)_{i,k} \otimes (D_P)_{k,j} = \bigoplus_k ((d_k)_i + (d_j)_k) \leq (d_j)_i = (D_P)_{i,j}.$$
On the other hand, since the diagonal entries of $D_P$ are all 0, it is clear that $(D_P \otimes D_P)_{i,j} \geq (D_P)_{i,j}$, so that $D_P$ is idempotent. Thus we have shown that $D_P$ is a max-plus idempotent matrix with 0's on the diagonal; in other words, $D_P$ is a max-plus Kleene star.
\end{proof}

We consider next the important special case where $P$ is both
max-plus and min-plus convex, and therefore has both a min-plus and a
max-plus dominator matrix.

\begin{theorem}\label{thm:domrelation}
Let $P \subseteq \mathbb{R}^n$ be both max-plus and min-plus convex with
compact projectivisation. Then $P$ is both a max-plus polytope
and a min-plus polytope, and its max-plus dominator is the negated
transpose of its min-plus dominator.
\end{theorem}
\begin{proof}
That $P$ is both a max-plus polytope and a min-plus polytope follows
from Proposition~\ref{prop:dommat} and its dual.

Let $D_P$ be the min-plus dominator of $P$ considered as a max-plus convex set.
We aim first to show that $P$ is the min-plus column space of $-D_P^T$.
It follows from Proposition~\ref{prop:dommat} that $D_P$ is a max-plus
Kleene star with max-plus column space $P$. By Lemma~\ref{lem:dual}(i), $-D_P$ is
a min-plus Kleene star and it follows easily that $-D_P^T$ is also a
min-plus Kleene star. Moreover, by Lemma~\ref{lem:dual}(iv), the min-plus
column space of $-D_P^T$ is exactly $P$.

Now let $D_P'$ be the max-plus dominator of $P$. Using the fact $P$ is
both min-plus and max-plus convex, it follows from the dual to
Proposition~\ref{prop:dommat} that $D_P'$ is a min-plus Kleene star
with min-plus column space $P$.

However, it follows from a result of Sergeev \cite[Proposition 6]{Sergeev07}
that there is at most one Kleene star with a given column
space, so we conclude that $D_P' = -(D_P)^T$ as required.
\end{proof}

\begin{corollary}
Let $P \subseteq \mathbb{R}^n$ be max-plus and min-plus convex with compact
projectivisation, and let $D_P$ be the min-plus dominator matrix of $P$. Then
\begin{itemize}
\item [(i)] $P$ is a max-plus polytope generated by the columns of $D_P$.
\item [(ii)] $P$ is a min-plus polytope generated by the rows of $-D_P$.
\end{itemize}
\end{corollary}

\section{Proofs of the main theorems}
\label{sect:maintheorems}

In this section we apply the technical results so far to establish our main
theorems, as described in the introduction.

\begingroup
\setcounter{theorem}{0} 
\renewcommand\thetheorem{\Alph{theorem}}
\begin{theorem}
Let $P\subseteq \mathbb{R}^n$ be max-plus convex with compact projectivisation.
Then the following are equivalent:
\begin{itemize}
\item[(i)] $P$ is min-plus convex;
\item[(ii)] $P$ is the max-plus column space of a max-plus Kleene star;
\item[(iii)] $P$ is the min-plus column space of a min-plus Kleene star.
\end{itemize}
\end{theorem}

\begin{proof}[Proof of Theorem \ref{thm:mainresult}]
Suppose first that $P$ is min-plus convex. Then by
Proposition \ref{prop:dommat} we see that $P= {\rm Col}_\oplus(D_P)$,
where $D_P$ is a max-plus Kleene star, establishing that (i) implies (ii).

It follows from Lemma \ref{lem:dual} that conditions (ii) and (iii) are equivalent. (Indeed, if $P$ is the max-plus column space of a max-plus Kleene star $K$, then it is the min-plus column space of the min-plus Kleene star formed by taking the negated transpose of $K$ and vice versa.)

Finally, that (ii) implies (i) is immediate from Lemma~\ref{lem:Kleene}.
\end{proof}

\begin{theorem}
Let $P \subseteq \mathbb{R}^n$. If any two of the following statements
hold, then so does the remaining one:
\begin{itemize}
\item[(i)] $P$ is a max-plus polytope;
\item[(ii)] $P$ is a min-plus polytope;
\item[(iii)] $P$ is Euclidean convex.
\end{itemize}
\end{theorem}
\endgroup

\begin{proof}
We shall show that (i) and (ii) together imply (iii), and that
(i) and (iii) together imply (ii). The remaining required implication
--- that (ii) and (iii) together imply (i) --- is dual to the latter
case by exchanging min and max.

Suppose first that (i) and (ii) hold. Then by Theorem A, $P$ is the
max-plus column space of a max-plus Kleene star and hence by Lemma 2.1 (or \cite[Propositions~2.6 and 3.1]{Sergeev09}) $P$ is Euclidean convex.

Now suppose (i) and (iii) hold, and let
$V=\{v_1, \ldots, v_m\}$ be a weak max-plus basis for $P$. Using
the type decomposition (see \cite{Develin04}) of $P$ with respect
to $V$ we note that $P$ is the union of finitely many bounded cells
$X_S$, each of which is compact, max-plus convex and
min-plus convex. We shall show that given any finite
collection of compact min-plus convex subsets of $\mathbb{R}^n$, if
the union is Euclidean convex then it is also min-plus convex. It
therefore follows from this result that $P$ is min-plus convex, and
hence by application of Theorem A we see that $P$ is indeed a
min-plus polytope.

Let $X_1,\ldots, X_k$ be a finite collection of compact, min-plus convex
subsets of $\mathbb{R}^n$ whose union (call it $P$) is Euclidean convex.
Suppose for a contradiction that $P$ is not min-plus convex. Clearly $P$
is closed under tropical scaling, so we must be able to choose
$x, y \in P$ with $x \boxplus y \notin P$. Write $q = x \boxplus y$.
Since $P$ is compact we may assume without loss of generality that $x$
is minimal in $P$ with $x \boxplus y = q$, in the weak sense that there
is no $x' \in P$ with $q < x' < x$. (Indeed, having fixed $y$ and $q$,
choose $x \in P$ with $x \boxplus y = q$ such that the first coordinate
is minimum possible, the second coordinate minimum possible subject to
the first and so forth.)

Now consider the Euclidean line segment
from $x$ to $y$. Note that this is a proper line segment, since $x = y$
would imply $q \in P$, giving a contradiction. Since $P$ is Euclidean-convex,
the line segment
is contained
in $P$. Since $P$ is the union of the (finitely many) $X_i$s and are all
closed, it is easy to see that there must
be some $i$ such that $X_i$ contains both $x$ and some non-end-point (call
it $z$) of the
line segment. (Indeed, if we consider a sequence of non-end-points
on the line segment converging to $x$, there must be some $X_i$ which contains
infinitely many terms, hence which contains a sequence converging to $x$,
hence which contains $x$.) Notice that $z$ is a weighted average
of $x$ and $y$, both of which are larger than $q$, so it must itself be
larger than $q$.

Now consider the point $x \boxplus z$. This lies in $X_i$, since the
latter is min-plus convex, and hence in $P$. Since $x$ and $z$ are both
larger than $q$, we have $x \boxplus z \geq q$, and by definition
$x \boxplus z \leq x$. By the minimality assumption
on $x$, this means that we must have $x \boxplus z = x$, in other
words, $x \leq z$.

But $z$ is a weighted average of $x$ and $y$, so this implies
$x \leq y$, whereupon $q = x \boxplus y = x \in P$, giving a
contradiction.
\end{proof}

Notice that if $P \subseteq \mathbb{R}^n$ is max-plus convex with compact
projectivisation and satisfies any of
 the equivalent conditions of Theorem \ref{thm:mainresult}, then $P$
must be Euclidean convex (this can be seen using Lemma \ref{lem:Kleene},
for instance). We remark that there are however max-plus convex sets in
$\mathbb{R}^n$ which have compact projectivisation and are Euclidean
convex, yet not min-plus convex. For example, consider the max-plus
convex hull $P \subseteq \mathbb{R}^3$ of the points
$$\{(1,0,0), (a,-a,0): 0 \leq a \leq 1\}.$$
It is easy to see that the projectivisation of $P$ is a Euclidean
triangle in $\mathbb{R}^2$. However, the min-plus sum of
any two distinct elements of the form $(a,-a,0)$ lies outside
$P$. Note that this means $P$ cannot be a max-plus polytope: indeed, if it
were the convex hull of finitely many of its elements then by Theorem~B it
would have to be also a min-plus polytope.

\bibliographystyle{plain}

\def\cprime{$'$} \def\cprime{$'$}

\end{document}